\documentclass[11pt,titlepage]{article}
\usepackage{amsthm}
\newtheorem{lemma}{Lemma}

\newtheorem{theorem}{Theorem}
\newtheorem{corollary}{Corollary}

\newtheorem{remark}{Remark}

\usepackage[latin1]{inputenc}
\usepackage{rotating, rotfloat}
\usepackage[round]{natbib} 
\usepackage{graphics, graphicx, hyperref, url, amsmath, subfigure, amsthm, amsfonts, enumerate, epstopdf, booktabs, bm}
\hypersetup{
    pdfstartview={FitH},    
    colorlinks=true,        
    linkcolor=blue,         
    citecolor=blue,         
    filecolor=blue,         
    urlcolor=blue           
}

  \def\0{\boldsymbol 0}

\def\etal{\itshape et al.}

\newcommand{\bec}{\begin{center}}
\newcommand{\enc}{\end{center}}
\newcommand{\bee}{\begin{eqnarray*}}
\newcommand{\ene}{\end{eqnarray*}}
\newcommand{\beq}{\begin{equation}}
\newcommand{\eeq}{\end{equation}}

\def\P{ {\sf P}}                  \def\E{ {\sf E}\,}
                   
\newcommand{\lra}[1]{\left( #1 \right)}

\newcommand{\lrd}[1]{\left\lfloor #1 \right\rfloor}
\def\ll{\left}
\def\rr{\right}

\def\be{\begin{eqnarray}}
\def\en{\end{eqnarray}}
\def\bes{\begin{eqnarray*}}
\def\ens{\end{eqnarray*}}
\def\toind{\overset{d}{\longrightarrow}}

\begin{document}

\title{\bf Bahadur representations for the bootstrap median absolute deviation and the application to\\ projection depth weighted mean}
\vskip 5mm

\author {{Qing Liu$^{a, b}$,
        Xiaohui Liu$^{a, b}$   \footnote{Corresponding author's email: csuliuxh912@gmail.com.},
        Zihao Hu$^{a, b}$}\\ \\[1ex]
        {\em\footnotesize $^a$ School of Statistics, Jiangxi University of Finance and Economics, Nanchang, Jiangxi 330013, China}\\
        {\em\footnotesize $^b$ Research Center of Applied Statistics, Jiangxi University of Finance and Economics, Nanchang,}\\ {\em\footnotesize Jiangxi 330013, China}\\
}

\date{}
\maketitle

\begin{center}
{\sc Abstract}
\end{center}
Median absolute deviation (hereafter MAD) is known as a robust alternative to the ordinary variance. It has been widely utilized to induce robust statistical inferential procedures. In this paper, we investigate the strong and weak Bahadur representations of its bootstrap counterpart. As a useful application, we utilize the results to derive the weak Bahadur representation of the bootstrap sample projection depth weighted mean---a quite important location estimator depending on MAD.
\vspace{2mm}

\vspace{2mm}


\setlength{\baselineskip}{1.5\baselineskip}

\section{Introduction}
\paragraph{}
\label{Introduction}

Let $F$ be the distribution function of $X$. The related median $v = \text{Med}(X)$ is then defined as $F^{-1}(1/2) = \inf\{x: F(x)\geq 1/2\}$ which satisfies
\be
 \label{eq-1107-1}
    F(v-)\leq 1/2 \leq F(v).
\en
Suppose $X_1, X_2, \cdots, X_n \overset{iid}{\sim} F$ and let $X_{1:n}, X_{2:n}, \cdots X_{n:n}$ be the related order statistics. The sample median is usually defined as
\bes
    \textrm{Med}_{n} = \frac{X_{\lrd{\frac{n+1}{2}}:n} + X_{\lrd{\frac{n+2}{2}}:n}}{2},
\ens
where $\lrd{\cdot}$ denotes the floor function. In the literature, the sample median is known as its high robustness properties and usually serves as an alternative to the sample mean in the location setting \citep{Sma1990}.

Based on $\textrm{Med}_{n}$ above, the sample MAD is defined as
\be
\label{eqn:MAD}
    \textrm{MAD}_{n} = \frac{W_{\lrd{\frac{n+1}{2}}:n} + W_{\lrd{\frac{n+2}{2}}:n}}{2},
\en
where $W_{i:n}$, $i = 1, 2, \cdots, n$, denote the order statistics related to $W_1 = |X_1 - \textrm{Med}_{n}|, W_2 = |X_2 - \textrm{Med}_{n}|, \cdots, W_n = |X_n - \textrm{Med}_{n}|$. Clearly, the population version, say $\xi$, of $\textrm{MAD}_{n}$ is the median of the distribution $G$ of $|X-v|$, i.e.,
\be
\label{eq-1107-2}
    G(y)=\P(|X-v|\leq y)=F(v+y)- F(v-y-),~ y\in \mathbb{R}.
\en

Similar to the sample median, $\textrm{MAD}_{n}$ is a famous robust scatter measure and hence a desirable alternative to the sample variance when outliers are present \citep{MS2009}. They together are widely used in statistics to construct some statistical inferential procedures, which have high breakdown point robustness. Among them, one famous example is the projection depth studied by \cite{Liu1992, Zuo2003}, which depends on a combination of one location estimator and one scale estimator with the most commonly used combination being (Med, MAD). Based on the projection depth, a few desirable estimators, as well as some inferential procedures, have been developed in the past decades; see, e.g., \cite{Zuo2003, ZCH2004, Zuo2006, DG2012} and references therein for details.

One well-known projection depth based estimator is the projection depth weighted mean, which includes the famous Stahel-Donoho estimator as a special case \citep{Don1982, Sta1981}. It turns out that this estimator enjoys very high efficiency and robustness \citep{ZCH2004}. Especially, it is interesting to find by \cite{Zuo2010} that combining the projection depth weighted mean with the bootstrap procedure, it is possible to construct a confidence interval \emph{which is even more optimal than the classical $t$ confidence interval} in the sense of having better finite sample performance. Nevertheless, the good property of the bootstrap sample projection depth weighted mean of \cite{Zuo2010} was only confirmed by some simulated examples, having no theoretical argument related to its limit distribution as far to the best of our knowledge. This motivates us to conduct the current research.

To achieve this, we need first to investigate the asymptotic properties of the related bootstrap median and bootstrap MAD. In the literature, it is known that the Bahadur representation, named after \cite{Bah1966}, is a useful tool to study the asymptotic properties of an estimator, because it provides not only an approximation to the estimator in the form of a sum of independent variables, but also a higher-order remainder from which one can see the convergence rate of the estimator as the sample size $n$ increases. Much attention has been paid to this tool since its introduction; see, e.g. \cite{Kie1967, HS1996, Wu2005, Wen2011} for details. Recently, \cite{MS2009} considered the Bahadur representation of the sample MAD, and \cite{Zuo2015} considered the Bahadur representations of the bootstrap sample quantiles.

In view of this, we will first consider the Bahadur representations for the \emph{bootstrap sample MAD}, and then apply the results to the case of \emph{bootstrap sample projection depth weighted mean}. Given the random sample $X_{1},X_{2},\ldots, X_{n}$ above, let $X^{*}_{1},X^{*}_{2},\ldots, X^{*}_{n}$ be the bootstrap sample from its empirical function $F_{n}$. Hereafter, denote $F^{*}_{n}$, $\textrm{Med}_{n}^*$ and $\textrm{MAD}_{n}^*$ as the empirical function, median and MAD corresponding to $X^{*}_{1},X^{*}_{2},\ldots, X^{*}_{n}$, respectively.

 Although the definition of MAD is essentially a quantile of the absolute deviation values, the result of \cite{Zuo2015} \emph{cannot be trivially applied directly to} the bootstrap sample MAD, as well as the bootstrap sample projection weighted mean, since it involves in the sample median, which depends on all of the bootstrapping observations; see \eqref{eqn:MAD}. 

For simplicity, we introduce some frequently used notations before starting the discussions. For any random event $A$, denote $\P^{*}(A)=\P(A|X_{1},X_{2}, \ldots, X_{n})$, i.e., the conditional probability. 
$\epsilon$ is any given positive constant, whose value may be not the same at different places. The term `\emph{a.s.}' stands for `almost surely'. For any \emph{fixed} integers $l$ and $m$ such that $\lrd{\frac{n}{2}} \ge l\geq 1$ and $\lrd{\frac{n}{2}} \ge m\geq 1$, denote $\hat{v}_{n,l}= X_{\ll\lfloor\frac{n+l}{2}\rr\rfloor:n}$, $\hat{\xi}_{n,m,l} = W_{\ll\lfloor\frac{n+m}{2}\rr\rfloor :n, l}$, with $W_{1:n, l} \leq \ldots \leq W_{n:n, l}$ the ordered statistics of $W_{i, l}=|X_i-\hat{v}_{n,l}|$, $1\leq i\leq n$. Their bootstrap counterparts will be denoted by $\hat{v}^{*}_{n,l}= X^{*}_{\ll\lfloor\frac{n+l}{2}\rr\rfloor:n}$ and $\hat{\xi}^{*}_{n,m,l}= W^{*}_{\ll\lfloor\frac{n+m}{2}\rr\rfloor :n, l}$, respectively. Without confusion, we assume that all $l$ and $m$ are fixed and satisfy $\lrd{\frac{n}{2}} \ge l, m \ge 1$ in the sequel.

The rest of this paper is organized as follows. Section \ref{Sec:strong} states the strong Bahadur representation of the bootstrap sample MAD, while its weak Bahadur representation is given in Section \ref{Sec:weak}. 
Based on these representations and the result of \cite{Zuo2015}, we further derive the joint distribution of the bootstrap sample median and MAD in Section \ref{Sec:joint}. As an application, we employ these results to further derive the weak Bahadur representation, as well as the limit distribution, of the bootstrap projection depth weighted mean. Some concluding remarks end this paper.

\section{Strong Bahadur representation for the bootstrap MAD}\label{Sec:strong}
\paragraph{}

In this section, we consider the strong Bahadur representation for the bootstrap sample MAD under a twice differentiable condition. Similar representations for the sample MAD can be found in \cite{MS2009}. Before proceeding to the derivation of the main result, we need several preliminary lemmas as follows.

\begin{lemma}(Hoeffding; see \cite{Ser1980})
  Let $Y_1, Y_2, \cdots, Y_n$ be independent random variables satisfying $P(a \le X_i \le b) = 1$, each $i$, where $a < b$. Then for $t > 0$,
  \bes
    P\ll(\sum_{i=1}^n (Y_i - E(Y_i)) \ge nt \rr) \le e^{-2nt^2 / (b - a)^2}.
  \ens
\end{lemma}

The detailed proof of the Hoeffding inequality was given in \cite{Ser1980}. Relying on it, we are able to show the following useful probability inequalities given in Lemmas \ref{lem:01}-\ref{lm-1107-2}.

\begin{lemma}\label{lem:01}
\label{lm-1107-1}
Let $v= F^{-1}(1/2)$ be the unique solution to \eqref{eq-1107-1}, then for any $\epsilon>0$, fixed integer $l$ and sufficiently large $n$, we have
\bes
    \P\ll(|\hat{v}^{*}_{n,l}-v|>\epsilon\rr)\leq 2e^{-\sqrt{2}n\delta^{2}_{\epsilon,n}},
\ens
where $\delta_{\epsilon,n}= \min\{a_{0,l}, b_{0,l}\}$ with
\be
\label{eq-1107-3}
    a_{0,l}: = a_{0}(\epsilon, l)= F\ll(v+\frac{\epsilon}{2}\rr) - \ll(\ll\lfloor\frac{n+l}{2}\rr\rfloor - 1\rr)\Big/n,
\en
\be
\label{eq-1107-4}
    b_{0,l}: = b_{0}(\epsilon,l)= \ll\lfloor\frac{n+l}{2}\rr\rfloor \Big/n - F\ll(v-\frac{\epsilon}{2}\rr).
\en
\end{lemma}

\begin{proof}[Proof of Lemma \ref{lem:01}]
By Hoeffding's inequality, we have
\begin{eqnarray*}
    & &\P^{*}\ll(\hat{v}^{*}_{n,l} > v+\frac{\epsilon}{2}\rr)\\
    & &=\P^{*}\left(nF^{*}_{n}\ll(v+\frac{\epsilon}{2}\rr)\leq \ll\lfloor\frac{n+l}{2}\rr\rfloor-1\right)\\
    & &=\P^{*}\left(\sum_{i=1}^{n} I\ll(X^{*}_{i}>v+\frac{\epsilon}{2}\rr) \geq n - \ll(\ll\lfloor\frac{n+l}{2}\rr\rfloor-1\rr)\right)\\
    & &= \P^{*}\left(\sum_{i=1}^{n}I\ll(X^{*}_{i}>v+\frac{\epsilon}{2}\rr)- \sum_{i=1}^{n} \ll(1 - F_{n}\ll(v+\frac{\epsilon}{2}\rr)\rr)
         \geq nF_{n} \ll(v+\frac{\epsilon}{2}\rr) - \ll(\ll\lfloor\frac{n+l}{2}\rr\rfloor-1\rr)\right)\\
    & & \leq \exp\left\{-2n\left(F_{n}(v+\frac{\epsilon}{2}) - \frac{\ll\lfloor\frac{n+l}{2}\rr\rfloor-1}{n}\right)^{2}\right\}.
\end{eqnarray*}
By the Glivenko-Cantelli theorem, $F_{n}(v+\frac{\epsilon}{2})\rightarrow F(v+\frac{\epsilon}{2}), a.s.$ Thus for sufficiently large $n$,
\bes
    F_{n}\ll(v+\frac{\epsilon}{2}\rr)-\frac{\ll\lfloor\frac{n+l}{2}\rr\rfloor-1}{n} ~ > ~\left(F\ll(v+\frac{\epsilon}{2}\rr)-\frac{\ll\lfloor\frac{n+l}{2}\rr\rfloor-1}{n}\right)\Big/\sqrt[4]{2} ~ > ~0, \quad a.s.
\ens
by noting that $F(v+\frac{\epsilon}{2})>\frac12$ for any $\epsilon>0$. Hence
\bes
    \P(\hat{v}^{*}_{n,l} > v+\frac{\epsilon}{2}) =\E\left[\P^{*}(\hat{v}^{*}_{n,l} > v+\frac{\epsilon}{2})\right] \leq e^{-\sqrt{2}n a_{0}^2(\epsilon ,l)}.
\ens
Similar discussion leads to
\bes
    \P(\hat{v}^{*}_{n,l} < v-\frac{\epsilon}{2})
    =\E\left[\P^{*}\ll(\hat{v}^{*}_{n,l} < v - \frac{\epsilon}{2}\rr)\right]
    \leq e^{-\sqrt{2}n b_{0}^2(\epsilon ,l)}.
\ens
Then the conclusion follows.
\end{proof}

\begin{lemma}
\label{lm-1107-2}
Suppose $v= F^{-1}(1/2)$ be the unique solution to \eqref{eq-1107-1}, and
$\xi =G^{-1}(1/2)$ is the unique solution to \eqref{eq-1107-2}. For fixed
$l$, $m$ and any $\epsilon>0$, when $n$ is sufficiently large, we have
\bes
    \P(|\hat{\xi}^{*}_{n,m,l} - \xi| > \epsilon) \leq 6e^{-\sqrt{2}n\Delta^{2}_{\epsilon,n}},
\ens
where $\Delta_{\epsilon,n} = \min\{a_{0,l}, b_{0,l}, c_{0,m}, d_{0,m}\}$ with $a_{0,l}, b_{0,l}$
defined in \eqref{eq-1107-3} and \eqref{eq-1107-4}, and $c_{0,m}, d_{0,m}$ defined by
\begin{eqnarray*}
    c_{0,m} := c_{0}(\epsilon, m)= F\ll(v+\xi+\frac{\epsilon}{2}\rr) - F\ll(v-\xi-\frac{\epsilon}{2}\rr)
                     -(\ll\lfloor\frac{n+m}{2}\rr\rfloor-1) \Big/ n,\\
    d_{0,m} := d_{0}(\epsilon, m)= \ll\lfloor\frac{n+m}{2}\rr\rfloor \Big/ n - F\ll(v+\xi-\frac{\epsilon}{2}\rr)
                     +F\ll(v-\xi+\frac{\epsilon}{2}\rr).
\end{eqnarray*}
\end{lemma}

\begin{proof}[Proof of Lemma \ref{lm-1107-2}]
Denote by $G^{*}_{n}(y)$ the empirical distribution function of $W^{*}_{1l},W^{*}_{2l},\ldots, W^{*}_{nl}$.
Set $\alpha_{n}=(\ll\lfloor\frac{n+m}{2}\rr\rfloor-1)/n$, then we have
\begin{eqnarray*}
    \P^{*}\ll(\hat{\xi}^{*}_{n,m,l}>\xi+\epsilon\rr)
    &=&\P^{*}\ll(W^{*}_{\ll\lfloor\frac{n+m}{2}\rr \rfloor l:n}>\xi+\epsilon\rr)\\
    &=& \P^{*}\left(G^{*}_{n}(\xi+\epsilon)\leq \alpha_{n}\right)\\
    &\leq & \P^{*}\left(G^{*}_{n}(\xi+\epsilon)\leq \alpha_{n},~ |\hat{v}^{*}_{n,l} - v|\leq\frac{\epsilon}2\right)
        +\P^{*}\ll(|\hat{v}^{*}_{n,l}-v|>\frac{\epsilon}2\rr)\\
    &\leq & \P^{*}\left(\sum_{i=1}^nI\ll(v-\xi-\frac{\epsilon}2 \leq X_i^* \leq v+\xi+\frac{\epsilon}2\rr) \leq n \alpha_{n} \right)
        +\P^{*}\ll(|\hat{v}^{*}_{n,l}-v|>\frac{\epsilon}2\rr),
\end{eqnarray*}
where the last inequality follows from
\begin{eqnarray*}
&&\ll\{\sum_{i=1}^nI(\hat{v}^{*}_{n,l}-\xi-\epsilon\leq X_i^*\leq \hat{v}^{*}_{n,l}+\xi+\epsilon)\leq n\alpha_n,~v-\frac\epsilon2\leq\hat{v}^{*}_{n,l}\leq v+\frac\epsilon2\rr\}\\
&&\subset\ll\{\sum_{i=1}^nI\ll(v-\xi-\frac{\epsilon}2 \leq X_i^* \leq v+\xi+\frac{\epsilon}2\rr) \leq n \alpha_{n},~v-\frac\epsilon2\leq\hat{v}^{*}_{n,l}\leq v+\frac\epsilon2\rr\}.
\end{eqnarray*}
For the first part, by Hoeffding's inequality, we have
\begin{eqnarray*}
    \lefteqn{\P^{*}\left(\sum_{i=1}^nI\ll(v-\xi-\frac{\epsilon}2 \leq X_i^*\leq v+\xi+\frac{\epsilon}2\rr) \leq n\alpha_{n}\right)}\\
    =&& \P^{*}\left(\sum_{i=1}^{n}I\ll(v-\xi-\frac{\epsilon}2 \leq X_i^*\leq v+\xi+\frac{\epsilon}2\rr) - np_{n}
         \leq n(\alpha_{n}-p_{n1})\right)\\
    =&& \P^{*}\left(\sum_{i=1}^{n}(Y_{i}-\E^* Y_{i})\geq n(p_{n1}-\alpha_{n})\right)\\
    \leq && \exp\{-2n(p_{n1}-\alpha_{n})^{2}\},
\end{eqnarray*}
where $Y_{i}=1-I\ll(v-\xi-\frac{\epsilon}2 \leq X^{*}_{i}\leq v+\xi+\frac{\epsilon}2\rr)$, $i=1,\ldots,n$ and
$p_{n1}=F_{n}\ll(v+\xi+\frac{\epsilon}2\rr) -F_{n}\ll(v-\xi-\frac{\epsilon}2-\rr)$.
Note that $G(\xi+\frac{\epsilon}2)=F\ll(v+\xi+\frac{\epsilon}2\rr) -F\ll(v-\xi-\frac{\epsilon}2-\rr)$, the Glivenko-Cantelli theorem implies that
\bes
    p_{n1}\rightarrow F\ll(v+\xi+\frac{\epsilon}2\rr) -F\ll(v-\xi-\frac{\epsilon}2-\rr) > \frac{1}{2}, \quad a.s.
\ens
Then for sufficiently large $n$
\bes
    p_{n1}-\alpha_n>\frac{F(v+\xi+\frac{\epsilon}2) -F(v-\xi-\frac{\epsilon}2)-(\ll\lfloor\frac{n+m}{2}\rr\rfloor-1)/n}{\sqrt[4]{2}}>0,\quad a.s.
\ens
It follows from the discussion above and Lemma \ref{lm-1107-1} that
\begin{eqnarray*}
    \P\ll(\hat{\xi}^{*}_{n,m,l}>\xi+\epsilon\rr)
    &=& \E\left[\P^{*}\ll(\hat{\xi}^{*}_{n,m,l}>\xi+\epsilon\rr)\right]\\
    &\leq&  2e^{-\sqrt{2}n\delta^{2}_{\epsilon,n}}
          + \E \left[e^{-2n(p_{n}-\alpha_{n})^{2}}\right]\\
    &\leq&  2e^{-\sqrt{2}n\delta^{2}_{\epsilon,n}}
          + e^{-\sqrt{2}nc_{0}^2(\epsilon,m)}\\
    &\leq& 3e^{-\sqrt{2}n\Delta^{2}_{\epsilon,n}}.
\end{eqnarray*}
Set $\beta_{n}=\ll\lfloor\frac{n+m}{2}\rr\rfloor/n$, a similar argument leads to
\bes
    \P^{*}(\hat{\xi}^{*}_{n,m,l} < \xi - \epsilon)
    &\leq & \P^{*}\left(\sum_{i=1}^nI\ll(v-\xi+\frac{\epsilon}2 \leq X_i^*\leq v+\xi-\frac{\epsilon}2\rr) \geq n\beta_{n}\right)
        +\P^{*}\ll(|\hat v^{*}_{n,l}-v|>\frac{\epsilon}2 \rr)\\
    &\leq & \exp\{-2n(\beta_{n}-p_{n2})^{2}\}+2e^{-\sqrt{2}n\delta^{2}_{\epsilon,n}},
\ens
where $p_{n2}=F_{n}(v+\xi-\frac{\epsilon}2) -F_{n}(v-\xi+\frac{\epsilon}2-)$, and $\beta_{n}-p_{n2}>0$ for sufficiently large $n$.
Then we have
\begin{eqnarray*}
    \P(\hat{\xi}^{*}_{n,m,l} < \xi-\epsilon)
    = \E\left[\P^{*}(\hat{\xi}^{*}_{n,m,l} < \xi-\epsilon)\right]
    \leq 3e^{-\sqrt{2}n\Delta^{2}_{\epsilon,n}}.
\end{eqnarray*}
The conclusion has been proved.
\end{proof}

\begin{lemma}\label{lm-1107-4}
Let F be differentiable at v and $v\pm \xi$, with $F'(v)>0$ and $G'(\xi)=F'(v-\xi)+F'(v+\xi)>0$, then for any fixed $l \ge 1$ and $m \ge 1$, we have almost surely
\bes
    |(\hat{v}^{*}_{n,l}-\hat{\xi}^{*}_{n,m,l})-(v-\xi)|\leq D \frac{(\log n)^{1/2}}{n^{1/2}}
\ens
and
\bes
    |(\hat{v}^{*}_{n,l} + \hat{\xi}^{*}_{n,m,l})-(v+\xi)|\leq D \frac{(\log n)^{1/2}}{n^{1/2}}
\ens
for sufficiently large n, where $D=\max\{8/F'(v),8/G'(\xi)\}$.
\end{lemma}

\begin{proof}[Proof of Lemma \ref{lm-1107-4}]
Put $\epsilon_n=D \frac{(\log n)^{1/2}}{n^{1/2}}$. It follows from Lemma \ref{lm-1107-1} and Lemma \ref{lm-1107-2} that, for any fixed $l \ge 1$ and $m \ge 1$,
\bes
    \P(|(\hat{v}^{*}_{n,l} + \hat{\xi}^{*}_{n,m,l})-(v+\xi)| > \epsilon_n) \leq 8\exp\{-\sqrt 2n\Delta^2_{\epsilon_n/2, n}\}.
\ens
Since $F(v)=1/2$, we have
\begin{eqnarray*}
    a_0\ll(\frac{\epsilon_n}2, l\rr)
    &=& F\ll(v+\frac{\epsilon_n}4\rr)-\frac{\ll\lfloor\frac{n+l}{2}\rr\rfloor-1}n\\
    &=& F\ll(v+\frac{\epsilon_n}4\rr)- \frac12+O\ll(\frac1n\rr)\\
    &=& F\ll(v+\frac{\epsilon_n}4\rr)- F(v)+O\ll(\frac1n\rr)\\
    &=& \frac{F'(v)}4 \epsilon_n +o\ll(\epsilon_n\rr)+O\ll(\frac1n\rr)\\
    &>& \frac{(\log n)^{1/2}}{n^{1/2}}, \quad\text{for sufficiently large } n.
\end{eqnarray*}
Similarly
\bes
    b_0\ll(\frac{\epsilon_n}2, l\rr) > \frac{(\log n)^{1/2}}{n^{1/2}}, \quad\text{for sufficiently large } n.
\ens
By similar arguments using $F(v+\xi)-F(v-\xi)=1/2$, we also obtain
\bes
    c_0\ll(\frac{\epsilon_n}2, m\rr) > \frac{(\log n)^{1/2}}{n^{1/2}}, \quad\text{for sufficiently large } n
\ens
and
\bes
    d_0\ll(\frac{\epsilon_n}2, m\rr) > \frac{(\log n)^{1/2}}{n^{1/2}}, \quad\text{for sufficiently large }n.
\ens
The conclusion follows from the inequalities above and the Borel-Cantelli lemma.
\end{proof}

\begin{lemma}
\label{lm-1107-5}
Let F be differentiable at v and twice differentiable at $v\pm \xi$, with $F'(v)>0$ and $G'(\xi)=F'(v-\xi)+F'(v+\xi)>0$, then for any fixed $l \ge 1$ and $m \ge 1$, as
$n\to \infty$, we have
\bes
    H_{n1}=\ll|F_n^*(\hat{v}^{*}_{n,l}-\hat{\xi}^{*}_{n,m,l}) - F_n^*(v-\xi)- F(\hat{v}^{*}_{n,l}- \hat{\xi}^{*}_{n,m,l}) + F(v-\xi)\rr| = O\ll(n^{-3/4} \log n\rr), \quad a.s.
\ens
and
\bes
    H_{n2}=\ll|F_n^*(\hat{v}^{*}_{n,l} + \hat{\xi}^{*}_{n,m,l}) - F_n^*(v+\xi) - F(\hat{v}^{*}_{n,l} + \hat{\xi}^{*}_{n,m,l}) + F(v+\xi)\rr| = O\ll(n^{-3/4} \log n\rr), \quad a.s.
\ens
\end{lemma}

\begin{proof}[Proof of Lemma \ref{lm-1107-5}]
Denote by $\theta_p$ the $p$-th quantile of $F$ for $p\in (0, 1)$. Let $a_n =\frac{c\log n}{n^{1/2}}$ for some positive constant $c$, and define
\bes
    H_{pn}(x):=[F^*_n(x)-F^*_n(\theta_p)]- [F(x)-F(\theta_p)].
\ens
It follows from Lemma 3.7 of \cite{Zuo2015} that
\bes
    \sup_{|x-\theta_p|<a_n}|H_{pn}(x)|=O\ll(n^{-3/4}\log n\rr), \text{as}\, n\to \infty, \quad a.s.
\ens
Let we express $v-\xi$ as the $p$-th quantile of $F$: $v-\xi=F^{-1}(p)=\theta_p$, and put $x_n=\hat{v}^{*}_{n,l} - \hat{\xi}^{*}_{n,m,l}$ for any fixed $l,m \ge 1$, then Lemma \ref{lm-1107-4} implies
\bes
    |x_n-\theta_p| \leq D \frac{(\log n)^{1/2}}{n^{1/2}} < a_n, \,\text{for sufficiently large }n.
\ens
Now we have
\bes
    H_{n1} \leq \sup_{|x-\theta_p|<a_n}|H_{pn}(x)|=O\ll(n^{-3/4}\log n\rr), \quad a.s.
\ens
Similarly we can obtain
\bes
    H_{n2} =O\ll(n^{-3/4}\log n\rr), \quad a.s.
\ens
The proof has been completed.
\end{proof}

\begin{lemma}
\label{lm-1107-3}
Suppose $v= F^{-1}(1/2)$ be the unique solution to \eqref{eq-1107-1}, and
$\xi =G^{-1}(1/2)$ is the unique solution to \eqref{eq-1107-2}. Then for any fixed $m \ge 1$, it holds almost surely that
\bes
    G_{n}^{*}\ll(\hat{\xi}^{*}_{n,m,l}\rr) = \frac{1}{2} + O\ll(\frac{\log n}{n}\rr), \quad n \rightarrow \infty.
\ens
\end{lemma}

\begin{proof}[Proof of Lemma \ref{lm-1107-3}]
For convenience we set $l=m=1$. Recall that $G^{*}_{n}(y)$ is the empirical distribution function of $W^{*}_{1,l},W^{*}_{2,l},\ldots, W^{*}_{n,l}$. Since $\hat{\xi}^{*}_{n,1,1} = W^{*}_{\ll\lfloor\frac{n+1}{2}\rr\rfloor :n, 1}$, we have $G_{n}^{*}(\hat{\xi}^{*}_{n,1,1})=\ll\lfloor\frac{n+1}{2}\rr\rfloor/n$ unless there is a tie. If such a tie exists, we have some $X_{i}^{*}=\hat{v}^{*}_{n,1} \pm \hat{\xi}^{*}_{n,1,1}$. It follows from \cite{Zuo2015} that for large $n$
\bes
    \sum_{i=1}^{n}I\lra{X_{i}^{*} = \hat{v}^{*}_{n,1}\pm\hat{\xi}^{*}_{n,1,1}} < 2\log n, \quad a.s.
\ens
That is, we have for large $n$
\bes
    nG_{n}^{*}(\hat{\xi}^{*}_{n,1,1})\leq \ll\lfloor\frac{n+1}{2}\rr\rfloor+ 2\log n,\quad a.s.
\ens
Then almost surely
\bes
    G_{n}^{*}(\hat{\xi}^{*}_{n,1,1})=\frac{1}{2}+ O\lra{\frac{\log n}{n}}, \quad n\rightarrow \infty.
\ens
This completes the proof of this lemma.
\end{proof}

After proving Lemmas \ref{lem:01}-\ref{lm-1107-3}, we now are able to show the following theorem, which states the strong Bahadur representation for $\hat{\xi}^{*}_{n,m,l}$ for any fixed $l, m \ge 1$.

\begin{theorem}\label{th:01}
Suppose $F$ is continuous in neighborhoods of $v\pm\xi$ and twice differentiable at $v$ and $v\pm \xi$, with $F'(v)>0$ and $G'(\xi)=F'(v-\xi)+F'(v+\xi)>0$, then for any fixed $m \ge 1$,
\bes
    \hat{\xi}^{*}_{n,m,l} - \xi = \frac{\frac12-[F^*_n(v+\xi)-F^*_n(v-\xi)]}{G'(\xi)}+\frac{F'(v+\xi)-F'(v-\xi)}{G'(\xi)}\frac{\frac12-F^*_n(v)}{F'(v)}+R_{n1}
\ens
with
\bes
    R_{n1}=O(n^{-3/4}\log n), \quad a.s.
\ens
\end{theorem}

\begin{proof}[Proof of Theorem \ref{th:01}]

It follows from Lemma \ref{lm-1107-5} that almost surely, for any fixed $l, m \ge 1$,
\bes
    F(\hat{v}^{*}_{n,l} + \hat{\xi}^{*}_{n,m,l}) - F(v+\xi) = F_n^*(\hat{v}^{*}_{n,l} + \hat{\xi}^{*}_{n,m,l}) - F_n^*(v+\xi)+O(n^{-3/4}\log n)
\ens
and
\bes
    F(\hat{v}^{*}_{n,l}-\hat{\xi}^{*}_{n,m,l})-F(v-\xi)=F_n^*(\hat{v}^{*}_{n,l} - \hat{\xi}^{*}_{n,m,l}) - F_n^*(v-\xi)+ O\ll(n^{-3/4}\log n\rr).
\ens
Taking the difference yields
\be
\label{eq-1107-5}
    &&F(\hat{v}^{*}_{n,l} + \hat{\xi}^{*}_{n,m,l}) - F(\hat{v}^{*}_{n,l} - \hat{\xi}^{*}_{n,m,l}) - F(v+\xi) + F(v-\xi) \\
    && = F_n^*(\hat{v}^{*}_{n,l} + \hat{\xi}^{*}_{n,m,l}) - F_n^*(\hat{v}^{*}_{n,l} - \hat{\xi}^{*}_{n,m,l})-F_n^*(v+\xi)+F_n^*(v-\xi) + O\ll(n^{-3/4}\log n\rr).\nonumber
\en
On the other hand, using Taylor expansion and Lemma \ref{lm-1107-4}, we have as $n \to \infty$
\bes
    F(\hat{v}^{*}_{n,l}+\hat{\xi}^{*}_{n,m,l})-F(v+\xi)=F'(v+\xi)(\hat{v}^{*}_{n,l}+\hat{\xi}^{*}_{n,m,l}-v-\xi)+O\ll(\frac{\log n}n \rr)
\ens
and
\bes
    F(\hat{v}^{*}_{n,l}-\hat{\xi}^{*}_{n,m,l})-F(v-\xi)=F'(v-\xi)(\hat{v}^{*}_{n,l}-\hat{\xi}^{*}_{n,m,l}-v+\xi)+O\ll(\frac{\log n}n \rr),
\ens
which implies
\be
\label{eq-1107-6}
    &&F(\hat{v}^{*}_{n,l}+\hat{\xi}^{*}_{n,m,l})-F(\hat{v}^{*}_{n,l}-\hat{\xi}^{*}_{n,m,l})-F(v+\xi)+F(v-\xi)\\
    &&= F'(v+\xi)(\hat{v}^{*}_{n,l}+\hat{\xi}^{*}_{n,m,l}-v-\xi)-F'(v-\xi)(\hat{v}^{*}_{n,l}-\hat{\xi}^{*}_{n,m,l}-v+\xi)+O\ll(\frac{\log n}n\rr)\nonumber.
\en
Combining (\ref{eq-1107-5}) and (\ref{eq-1107-6}) yields
\begin{eqnarray*}
    &&F'(v+\xi)(\hat{v}^{*}_{n,l}+\hat{\xi}^{*}_{n,m,l}-v-\xi)-F'(v-\xi)(\hat{v}^{*}_{n,l}-\hat{\xi}^{*}_{n,m,l}-v+\xi) \\
    && = F_n^*(\hat{v}^{*}_{n,l}+\hat{\xi}^{*}_{n,m,l})-F_n^*(\hat{v}^{*}_{n,l}-\hat{\xi}^{*}_{n,m,l})-F_n^*(v+\xi)+F_n^*(v-\xi)+O(n^{-3/4}\log n)\\
    &&= G_n^*(\hat{\xi}^{*}_{n,m,l})-[F_n^*(v+\xi)-F_n^*(v-\xi)]+O(n^{-3/4}\log n)\\
    &&= \frac12-[F_n^*(v+\xi)-F_n^*(v-\xi)]+O(n^{-3/4}\log n),
\end{eqnarray*}
where the last equality follows from Lemma \ref{lm-1107-3}. By noting that
\begin{eqnarray*}
    &&F'(v+\xi)(\hat{v}^{*}_{n,l}+\hat{\xi}^{*}_{n,m,l}-v-\xi)-F'(v-\xi)(\hat{v}^{*}_{n,l}-\hat{\xi}^{*}_{n,m,l}-v+\xi) \\
    &&= [F'(v+\xi)-F'(v-\xi)](\hat{v}^{*}_{n,l}- v) + G'(\xi)(\hat{\xi}^{*}_{n,m,l}-\xi),
\end{eqnarray*}
we have
\be
\label{eq-1107-07}
    \hat{\xi}^{*}_{n,m,l}-\xi &=& \frac{\frac12-[F^*_n(v+\xi)-F^*_n(v-\xi)]}{G'(\xi)}\\
    && + \frac{F'(v+\xi)-F'(v-\xi)}{G'(\xi)}(\hat{v}^{*}_{n,l}- v)
    +O(n^{-3/4}\log n).\nonumber
\en
Finally, taking $p=\frac12$ in Theorem 3.9 of \cite{Zuo2015} yields
\be
\label{eq-1107-08}
    \hat{v}^{*}_{n,l}= v+\frac{\frac12-F^*_n(v)}{F'(v)}+O(n^{-3/4}\log n).
\en
The proof is now completed by inserting (\ref{eq-1107-08}) into (\ref{eq-1107-07}).
\end{proof}

Since Theorem \ref{th:01} holds for any fixed $l, m \ge 1$, its result is quite general. Following a similar fashion to this theorem, it is easy to check the following theorem, which states the strong Bahadur representation for MAD$_n^*$.

\begin{theorem}\label{th:add1}
Under the conditions of Theorem \ref{th:01}, we have as $n\to\infty$
\begin{eqnarray*}
\text{MAD}_n^* - \xi &=& \frac{\frac12-[F^*_n(v+\xi)-F^*_n(v-\xi)]}{G'(\xi)}\\
&& +\frac{F'(v+\xi)-F'(v-\xi)}{G'(\xi)}\frac{\frac12-F^*_n(v)}{F'(v)}+O(n^{-3/4}\log n),~a.s.\end{eqnarray*}
\end{theorem}

\begin{proof}[Proof of Theorem \ref{th:add1}]
Observe that Med$_n^* = (\hat{v}^{*}_{n,1} + \hat{v}^{*}_{n,2}) / 2$, it is easy to verify that the result of Lemma \ref{lem:01} also holds for Med$_n^*$.
Let $\widetilde{\xi}^*_{n,m}= \widetilde W^*_{\ll\lfloor\frac{n+m}{2}\rr\rfloor:n}$, where $\widetilde W^*_{1:n}\leq\ldots\leq \widetilde W^*_{n:n}$ are the ordered statistics of $\widetilde W_i^*=|X_i^*-\text{Med}_n^*|$, $1\leq i\leq n$. Then by the same arguments, the results of Lemma \ref{lm-1107-2}-Lemma \ref{lm-1107-3} still hold with
$\hat{v}^{*}_{n,l}$ and $\hat\xi^*_{n,m,l}$ replaced by Med$_n^*$ and $\widetilde{\xi}^*_{n,m}$, respectively. Following the proof of Theorem \ref{th:01}, we have
\be
\widetilde{\xi}^{*}_{n,m}-\xi &=& \frac{\frac12-[F^*_n(v+\xi)-F^*_n(v-\xi)]}{G'(\xi)}\nonumber\\
    &&  +\frac{F'(v+\xi)-F'(v-\xi)}{G'(\xi)}\frac{\frac12-F^*_n(v)}{F'(v)}+O(n^{-3/4}\log n),~a.s. \nonumber
\en
Hence the conclusion follows by noting that MAD$_n^* = (\widetilde{\xi}^{*}_{n,1} + \widetilde{\xi}^{*}_{n,2}) / 2$.
\end{proof}

\section{Weak Bahadur representation for the bootstrap MAD}\label{Sec:weak}
\paragraph{}

The strong Bahabar representation is somewhat too strong. In statistics, deriving the weak Bahadur representation may suffice for many practical applications, such as deriving the limit distribution. Hence, in this section, we also consider the weak Bahadur representation of the bootstrap MAD under weaker conditions than Section \ref{Sec:strong}.

To achieve this, we first present some useful preliminary lemmas as follows.

\begin{lemma}
\label{lm-1107-6}
Let F be differentiable at v and $v\pm \xi$, with $F'(v)>0$ and $G'(\xi)=F'(v-\xi)+F'(v+\xi)>0$, then for any fixed $l,m \ge 1$, we have as $n\to\infty$
\bes
    |(\hat{v}^{*}_{n,l}-\hat{\xi}^{*}_{n,m,l})-(v-\xi)|=O_p(n^{-1/2})\quad \text{and} \quad |(\hat{v}^{*}_{n,l}+\hat{\xi}^{*}_{n,m,l})-(v+\xi)|=O_p(n^{-1/2}).
\ens
\end{lemma}

\begin{proof}[Proof of Lemma \ref{lm-1107-6}]
For any $\epsilon>0$, let $M > \sqrt {\log (1/\epsilon)}/\sqrt[4]{2}$. Put $\epsilon_n= D\frac{M}{n^{1/2}}$, where the constant $D$ is defined in Lemma \ref{lm-1107-4}. It can be seen from Lemmas \ref{lm-1107-1}-\ref{lm-1107-2} that
\bes
    \P(|(\hat{v}^{*}_{n,l}+\hat{\xi}^{*}_{n,m,l})-(v+\xi)|>\epsilon_n)\leq 8\exp\{-\sqrt 2 n\Delta^2_{\epsilon_n/2, n}\}.
\ens
Similar to Lemma \ref{lm-1107-4}, we have
\begin{eqnarray*}
    a_0\ll(\frac{\epsilon_n}2, l\rr)
    &=& F\ll(v+\frac{\epsilon_n}4\rr)-\frac{\ll\lfloor\frac{n+l}{2}\rr\rfloor-1}n\\
    &=& F\ll(v+\frac{\epsilon_n}4\rr)- F(v)+O\ll(\frac1n\rr)\\
    &=& \frac{F'(v)}4 \epsilon_n +o(\epsilon_n)+O\ll(\frac1n\rr)\\
    &>& \frac{M}{n^{1/2}}, \quad\text{for all sufficiently large } n.
\end{eqnarray*}
The same results hold for $b_0(\frac{\epsilon_n}2, l),c_0(\frac{\epsilon_n}2, m)$ and $d_0(\frac{\epsilon_n}2 ,m)$.
Now we have
\bes
    \sqrt 2 n\Delta^2_{\epsilon_n/2, n} \geq \sqrt 2 M^2\quad\text{for all sufficiently large } n,
\ens
whence for $n$ large enough
\bes
    \P(n^{1/2}|(\hat{v}^{*}_{n, l}+\hat{\xi}^{*}_{n,m,l})-(v+\xi)|>DM) \leq e^{-\sqrt2 M^2}<\epsilon,
\ens
which implies
\bes
    |(\hat{v}^{*}_{n, l}+\hat{\xi}^{*}_{n, m,l})-(v+\xi)|=O_p(n^{-1/2}).
\ens
The rest part can be proved by using the same steps.
\end{proof}

\begin{lemma}(Ghosh, 1971)
\label{lm-1107-7}
Let $\{U_n\}$ and $\{V_n\}$ be sequences of random variables on some probability space $(\Omega, \mathcal{F}, \P)$.
Suppose that (a) $V_n=O_p(1), n\to \infty$, and (b) For all $t$ and all $\epsilon>0$,
\begin{equation}
\begin{split}
\lim_{n\to\infty}\P(U_n\geq t+\epsilon, V_n\leq t)=0\\
\lim_{n\to\infty}\P(U_n\leq t, V_n\geq t+\epsilon)=0.
\end{split}
\end{equation}
Then $U_n-V_n=o_p(1)$, $n\to\infty$.
\end{lemma}

\begin{lemma}
\label{lm-1107-8}
Let F be continuous in the neighborhoods of $v\pm \xi$, and differentiable at $v$ and $v\pm \xi$, with $F'(v)>0$ and $G'(\xi)=F'(v-\xi)+F'(v+\xi)>0$, then as $n\to \infty$, we have for any fixed $l, m \ge 1$
\bes
    H_{n1}=|F_n^*(\hat{v}^{*}_{n,l}-\hat{\xi}^{*}_{n,m,l}) - F_n^*(v-\xi)-F(\hat{v}^{*}_{n,l}-\hat{\xi}^{*}_{n,m,l})+F(v-\xi)|=o_p(n^{-1/2})
\ens
and
\bes
    H_{n2}=|F_n^*(\hat{v}^{*}_{n,l}+\hat{\xi}^{*}_{n,m,l}) - F_n^*(v+\xi)-F(\hat{v}^{*}_{n,l}+\hat{\xi}^{*}_{n,m,l})+F(v+\xi)|=o_p(n^{-1/2}).
\ens
\end{lemma}

\begin{proof}[Proof of Lemma \ref{lm-1107-8}]
Let
\bes
    U_n &=& n^{1/2}[F_n^*(\hat{v}^{*}_{n,l}+\hat{\xi}^{*}_{n,m,l})-F_n^*(v+\xi)]\\
    V_n &=& n^{1/2}[F(\hat{v}^{*}_{n, l}+\hat{\xi}^{*}_{n,m,l})-F(v+\xi)].
\ens
By Taylor expansion and Lemma \ref{lm-1107-6},
\bes
    F(\hat{v}^{*}_{n, l}+\hat{\xi}^{*}_{n, m,l})-F(v+\xi)=O(|\hat{v}^{*}_{n, l}+\hat{\xi}^{*}_{n, m,l}-v-\xi|)=O_p(n^{-1/2}),\, n\to \infty.
\ens
Thus $V_n$ satisfies $(a)$ of Lemma \ref{lm-1107-7}.

Consider the case $t>0$. Define the right limit as
\bes
    \beta:=\lim_{t\to 0^+}F^{-1}(F(v+\xi)+t/\sqrt n).
\ens
Since $F^{-1}$ may be not continuous at $F(v+\xi)$, there are two cases to consider. When $\beta=v+\xi$, using $F(x) < p$ if and only if $x<F^{-1}(p)$, we have
\begin{eqnarray}\label{eq-1830-1}
\begin{aligned}
    \{V_n\leq t\}
    =& \ll\{F(\hat{v}^{*}_{n, l}+\hat{\xi}^{*}_{n, m,l})-F(v+\xi)\leq \frac t{\sqrt n}\rr\}\\
    \subset& \ll\{F(\hat{v}^{*}_{n, l}+\hat{\xi}^{*}_{n, m,l})<F(v+\xi)+ \frac {t+\epsilon/2}{\sqrt n}\rr\}\\
   =& \ll\{\hat{v}^{*}_{n, l}+\hat{\xi}^{*}_{n, m,l} < F^{-1}\ll(F(v+\xi)+\frac {t+\epsilon/2}{\sqrt n}\rr)\rr\}\\
   \subset& \ll\{F_n^*(\hat{v}^{*}_{n, l}+\hat{\xi}^{*}_{n, m,l})\leq F_n^*(\eta_n(t))\rr\}
\end{aligned}
\end{eqnarray}
where
\bes
    \eta_n(t)=F^{-1}\left(F(v+\xi)+\frac {t+\epsilon/2}{\sqrt n}\right).
\ens
By (\ref{eq-1830-1}) and the expressions of $U_n$ and $V_n$, we have
\be
\label{eq-1107-9}
    \P(U_n\geq t+\epsilon, V_n \leq t)\leq \P\left(F_n^*(\eta_n(t))-F_n^*(v+\xi)\geq \frac{t+\epsilon}{\sqrt n}\right).
\en
Since $F$ is continuous at $v+\xi$, which implies that $F(\eta_n(t))-F(v+\xi)=\frac {t+\epsilon/2}{\sqrt n} >0$. Then for all $n$ sufficiently large
\bes
    p_n:=F_n(\eta_n(t))-F_n(v+\xi)>0,\quad a.s.
\ens
Then for a sufficiently large $n$, given $X_1,X_2,\ldots,X_n$, we have
\bes
    Z_n^*=:n\left(F_n^*(\eta_n(t))-F_n^*(v+\xi)\right) \sim \text{Binomial}(n, p_n).
\ens
By using the Chebyshev inequality, and noting that $E(p_n)=\frac{t+\epsilon/2}{\sqrt n}$, we have
\begin{eqnarray*}
    &&\P\left(F_n^*(\eta_n(t))-F_n^*(v+\xi)\geq \frac{t+\epsilon}{\sqrt n}\right)\\
    &&= \E\left[\P^*\left(F_n^*(\eta_n(t))-F_n^*(v+\xi)\geq \frac{t+\epsilon}{\sqrt n}\right)\right]\\
    &&=\E\left[\P^*\left(Z_n^*-np_n\geq \sqrt n(t+\epsilon)-np_n\right)\right]\\
    &&\leq\E\left[\P^*\left(|Z_n^*-np_n|\geq \frac\epsilon3\sqrt n\right)\right]\\
    &&\leq \E\left[\frac{9p_n(1-p_n)}{\epsilon^2}\right]\leq \frac {9(t+\epsilon/2)}{\sqrt n \epsilon^2}\to 0,\quad n\to \infty.
\end{eqnarray*}
Returning to (\ref{eq-1107-9}), the first condition in (b) of Lemma \ref{lm-1107-8} is established for $t>0$ and $\beta=v+\xi$.

When $t>0$ and $\beta>v+\xi$, let $\theta$ be any point in the open interval $(v+\xi, \beta)$.
As has been proved in Section \ref{Sec:strong} that $\hat{v}^{*}_{n, l}+\hat{\xi}^{*}_{n, m,l}\to v+\xi,\, a.s.$ which implies
$\P(\hat{v}^{*}_{n, l}+\hat{\xi}^{*}_{n, m,l}>\theta)\to 0$ and
\bes
    \P(U_n\geq t+\epsilon, V_n\leq t)=\P(U_n\geq t+\epsilon, V_n\leq t, \hat{v}^{*}_{n, l}+\hat{\xi}^{*}_{n, m,l}\leq \theta)+o(1), n\to\infty.
\ens
Since $\eta_n(t)\to \beta>\theta$, then for sufficiently large $n$
\begin{eqnarray*}
    \{V_n\leq t, \hat{v}^{*}_{n, l}+\hat{\xi}^{*}_{n, m,l}\leq \theta\}
    &\subset& \{F(\hat{v}^{*}_{n, l}+\hat{\xi}^{*}_{n, m,l})<F(v+\xi)+ \frac {t+\epsilon/2}{\sqrt n},~ \hat{v}^{*}_{n, l}+\hat{\xi}^{*}_{n, m,l}\leq \theta\}\\
    &\subset& \{\hat{v}^{*}_{n, l}+\hat{\xi}^{*}_{n, m,l}<F^{-1}(F(v+\xi)+\frac {t+\epsilon/2}{\sqrt n}), ~ \hat{v}^{*}_{n,l}+\hat{\xi}^{*}_{n, m,l}\leq \theta\}\\
    &\subset& \{F_n^*(\hat{v}^{*}_{n,l}+\hat{\xi}^{*}_{n, m,l})\leq F_n^*(\theta)\}.
\end{eqnarray*}
Then similar to (\ref{eq-1107-9}), we have
\be
    \P(U_n\geq t+\epsilon, V_n\leq t, \hat{v}^{*}_{n,l}+\hat{\xi}^{*}_{n, m,l}\leq \theta)
    \leq \P\left(F_n^*(\theta)-F_n^*(v+\xi)\geq \frac{t+\epsilon}{\sqrt n}\right).
\en
Note that by the definition of $\beta$ and $\theta$, almost surely there are no sample in the interval
$[v+\xi, \theta]$, hence no bootstrap sample in the same interval. So $F_n^*(\theta)-F_n^*(v+\xi)=0,\, a.s.$ Hence
\bes
    \P\left(F_n^*(\theta)-F_n^*(v+\xi)\geq \frac{t+\epsilon}{\sqrt n}\right)=0.
\ens
Thus we establish the first condition in $(b)$ of Lemma \ref{lm-1107-8} for $t>0$. The case $t\leq 0$ and the second condition of $(b)$ can be proved similarly. That is, we obtain $H_{2n}=o_p(n^{-1/2})$.

The proof of $H_{1n}=o_p(n^{-1/2})$ follows a similar fashion. We omit the details.
\end{proof}

Based on Lemma \ref{lm-1107-6} and Lemma \ref{lm-1107-8}, we have the following theorem.

\begin{theorem}\label{th:02}
Suppose $F$ is continuous in the neighborhoods of $v\pm \xi$, and differentiable at $v$ and $v\pm \xi$, with $F'(v)>0$ and
$G'(\xi)=F'(v-\xi)+F'(v+\xi)>0$, then as $n\to\infty$
\bes
    \hat{\xi}^{*}_{n,m,l} - \xi = \frac{\frac12-[F^*_n(v+\xi)-F^*_n(v-\xi)]}{G'(\xi)} + \frac{F'(v+\xi)-F'(v-\xi)}{G'(\xi)}\frac{\frac12-F^*_n(v)}{F'(v)}+R_{n2}
\ens
with
\bes
    R_{n2}=o_p(n^{-1/2}).
\ens
\end{theorem}

\begin{proof}[Proof of Theorem \ref{th:02}]
Similar to the proof of Theorem \ref{th:01}, it follows from Lemma \ref{lm-1107-3}, for any fixed $l, m \ge 1$, that
\be
\label{eq-1107-10}
    G_{n}^{*}(\hat{\xi}^{*}_{n,m,l})=\frac{1}{2}+o_p(n^{-1/2}),\quad n\rightarrow \infty.
\en
Following the same steps as the proof of Theorem 1, Lemma \ref{lm-1107-8}, Lemma \ref{lm-1107-6} and (\ref{eq-1107-10}) yield

\be
\label{eq-1107-11}
    \hat{\xi}^{*}_{n,m,l}-\xi &=& \frac{\frac12-[F^*_n(v+\xi)-F^*_n(v-\xi)]}{G'(\xi)} \\
    && +\frac{F'(v+\xi)-F'(v-\xi)}{G'(\xi)} (\hat{v}^{*}_{n,l}- v) +o_p(n^{-1/2}).\nonumber
\en
Note that Lemma 3.4 of \cite{Zuo2015} implies
\be
\label{eq-1107-12}
    \hat{v}^{*}_{n, l}= v+\frac{\frac12-F^*_n(v)}{F'(v)}+o_p(n^{-1/2}).
\en
The proof is now completed by inserting (\ref{eq-1107-12}) into (\ref{eq-1107-11}).
\end{proof}

Similar to the proof of Theorem \ref{th:add1}, by the same arguments of Lemma \ref{lm-1107-6}, Lemma \ref{lm-1107-8} and Theorem \ref{th:02},
we have the following weak Bahadur representation of bootstrap sample MAD.

\begin{theorem}\label{th:add2}
Under the conditions of Theorem \ref{th:02}, we have as $n\to\infty$
\begin{eqnarray*}
\text{MAD}_n^*- \xi &=& \frac{\frac12-[F^*_n(v+\xi)-F^*_n(v-\xi)]}{G'(\xi)} \\
&&+ \frac{F'(v+\xi)-F'(v-\xi)}{G'(\xi)}\frac{\frac12-F^*_n(v)}{F'(v)}+o_p(n^{-1/2}).
\end{eqnarray*}
\end{theorem}

\section{Joint asymptotic normality for the bootstrap median and MAD}\label{Sec:joint}
\paragraph{}

In this section, we consider the joint asymptotic normality of $(\text{Med}^{*}_{n}, \text{MAD}^{*}_{n})$.
As in \cite{Fal1997} and \cite{SM2009}, define $\alpha= F(v-\xi)+F(v+\xi)$, $\beta=F'(v-\xi)-F'(v+\xi)$,
and $\gamma=\beta^2+4(1-\alpha)\beta F'(v)$. We need the following lemma, which is Proposition A.1 in \cite{WC2009}.

\begin{lemma}
\label{lm-1107-9}
Let $\{V_i\}$ be a sequence of random variables, such that for some function $h$, as $n\to\infty$, $h(V_1,\ldots,V_n)\xrightarrow{d}\Theta$, where $\Theta$ has a distribution function $H$. If $\{U_i\}$ is a sequence of random variables such that
\bes
    \P(U_n-h(V_1,\ldots,V_n)\leq s|V_1,\ldots,V_n)\rightarrow F(s)
\ens
almost surely for all $s\in\mathbb{R}$, where $F$ is a continuous distribution function, then
\bes
    \P(U_n\leq t)\rightarrow (H*F)(t)
\ens
for all $t\in\mathbb{R}$, where "$*$" denotes the convolution operator.
\end{lemma}

Based on this lemma, we are now able to show the following theorem.

\begin{theorem}\label{th:03}
Suppose $F$ is continuous in the neighborhoods of $v\pm \xi$, and differentiable at $v$ and $v\pm \xi$, with $F'(v)>0$ and
$G'(\xi)=F'(v-\xi)+F'(v+\xi)>0$, then as $n\to\infty$
\bes
    \begin{pmatrix}
      \sqrt n (\text{Med}_n^* - v)\\[2ex]
      \sqrt n (\text{MAD}_n^* - \xi)
    \end{pmatrix}
     \overset{d}{\longrightarrow} N\left(\begin{pmatrix}
      0\\
      0
    \end{pmatrix},\, \Sigma\right)
\ens
where $\Sigma=(\sigma_{ij})_{2\times2}$ with
\bes
&&\sigma_{11}=\frac{1}{2F'(v)^2},\\
&&\sigma_{12}=\sigma_{21}=\frac1{2F'(v)G'(\xi)}\left(1-4F(v-\xi)+\frac{\beta}{F'(v)}\right),\\
&&\sigma_{22}=\frac1{2G'(\xi)^2}\left(1+\frac{\gamma}{F'(v)^2}\right).
\ens
\end{theorem}

\begin{proof}[Proof of Theorem \ref{th:03}]
For every vector $\bm{\lambda}=(\lambda_1,\lambda_2)^T$ such that $\bm{\lambda}^T\Sigma\bm{\lambda}>0$,
it suffice to show
\bes
    \bm{\lambda}^T\begin{pmatrix}
      \sqrt n (\text{Med}_n^* - v)\\[2ex]
      \sqrt n (\text{MAD}_n^* - \xi)
    \end{pmatrix}\toind N(0,\bm{\lambda}^T\Sigma\bm{\lambda}),\, n\to\infty.
\ens
Note that
\[ \bm{\lambda}^T\begin{pmatrix}
      \sqrt n (\text{Med}_n^* - v)\\[2ex]
      \sqrt n (\text{MAD}_n^* - \xi)
    \end{pmatrix}
=\bm{\lambda}^T\begin{pmatrix}
      \sqrt n (\text{Med}_n^* -\text{Med}_n)\\[2ex]
      \sqrt n (\text{MAD}_n^* - \text{MAD}_n)
    \end{pmatrix}
+\bm{\lambda}^T\begin{pmatrix}
      \sqrt n (\text{Med}_n - v)\\[2ex]
      \sqrt n (\text{MAD}_n - \xi)
    \end{pmatrix}.
\]
It follows from \cite{SM2009} that as $n\to\infty$
$$ h(X_1,\ldots,X_n)= \bm{\lambda}^T\begin{pmatrix}
      \sqrt n (\text{Med}_n - v)\\[2ex]
      \sqrt n (\text{MAD}_n - \xi)
    \end{pmatrix} \toind N\left(0,\frac12\bm{\lambda}^T\Sigma\bm{\lambda}\right).$$
By Lemma \ref{lm-1107-9}, we need only to show that as $n\to\infty$
\be
\label{eq-1107-13}
    \sup_{s\in \mathbb{R}}\left| \P^*\left( \frac{\sqrt n\left(\text{Med}_n^* -\text{Med}_n, ~\text{MAD}_n^* - \text{MAD}_n\right)\bm{\lambda}}
    {\sqrt{\bm{\lambda}^T\Sigma\bm{\lambda}/2}}\leq s\right)-\Phi(s)\right| \rightarrow 0,\quad a.s.
\en
where $\Phi$ is the distribution function of $N(0,1)$. By the weak Bahadur representations of $\text{Med}_n, \text{Med}_n^*, \text{MAD}_n$ and $\text{MAD}_n^*$, we have
\begin{equation*}
    \sqrt n (\text{Med}_n^* -\text{Med}_n)=\sqrt n\frac{F_n(v)-F_n^*(v)}{F'(v)} + o_p(1),
    \end{equation*}
    \begin{eqnarray*}
    \sqrt n (\text{MAD}_n^* - \text{MAD}_n)&=&\sqrt n \frac{F_n(v+\xi)-F_n(v-\xi)-F_n^*(v+\xi)+F^*_n(v-\xi)}{G'(\xi)} \\
    &&\quad -\frac{\beta}{G'(\xi)} \sqrt n\frac{F_n(v)-F_n^*(v)}{F'(v)} + o_p(1).
\end{eqnarray*}
Then the left of (\ref{eq-1107-13}) can be expressed as
\begin{eqnarray*}
    \lefteqn{\sup_{s\in \mathbb{R}}\left| \P^*\left(\sqrt n \frac{\bar Y_n^*-\bar Y_n}{\hat{\sigma}_n}
    \leq \frac{\sqrt{\bm{\lambda}^T\Sigma\bm{\lambda}/2}}{\hat{\sigma}_n}s\right)-\Phi(s)\right|}\\
    =&&\sup_{s\in \mathbb{R}}\left| \P^*\left(\sqrt n \frac{\bar{Y}_n^*-\bar{Y}_n}{\hat{\sigma}_n}
    \leq s\right)-\Phi\ll(\frac{\hat{\sigma}_n}{\sqrt{\bm{\lambda}^T\Sigma\bm{\lambda}/2}}s\rr)\right|\\
    \leq&&\sup_{s\in \mathbb{R}}\left| \P^*\left(\sqrt n \frac{\bar{Y}_n^*-\bar{Y}_n}{\hat{\sigma}_n}
    \leq s\right)-\Phi(s)\right| + \sup_{s\in \mathbb{R}}\left| \Phi\ll(\frac{\hat{\sigma}_n}{\sqrt{\bm{\lambda}^T\Sigma\bm{\lambda}/2}}s\rr)-\Phi(s)\right|,
\end{eqnarray*}
where $\bar Y_n^*=\frac 1n \sum_{i=1}^n Y_i^*$ with
\bes
    Y_i^*=-\frac{\lambda_2}{G'(\xi)}I(v-\xi<X_i^*\leq v+\xi)+\left(\frac{\lambda_2\beta}{G'(\xi)}-\lambda_1\right)I_{(X_i^*\leq v)},
\ens
and $\bar Y_n=\frac 1n \sum_{i=1}^n Y_i=\E(Y_1^*|X_1,\ldots,X_n)$,  $\hat\sigma_n^2=\text{var}(Y_1^*|X_1,\ldots,X_n)=\frac1n \sum_{i=1}^n(Y_i-\bar Y_n)^2$ with
\bes
    Y_i=-\frac{\lambda_2}{G'(\xi)}I(v-\xi<X_i\leq v+\xi)+\left(\frac{\lambda_2\beta}{G'(\xi)}-\lambda_1\right)I(X_i\leq v).
\ens
Since $Y_1^*,\ldots,Y_n^*$ are iid random variables given $X_1,\ldots,X_n$, it follows from Berry-Essen theorem that
\bes
\label{eq-1107-14}
    \sup_{s\in \mathbb{R}}\left| \P^*\left(\sqrt n \frac{\bar{Y}_n^*-\bar{Y}_n}{\hat{\sigma}_n}
    \leq s\right)-\Phi(s)\right|\leq \frac{33}{4}\frac{\sum_{i=1}^n|Y_i-\bar Y_n|^3}{n^{3/2}\hat{\sigma}_n^3}.
\ens
Note that $Y_i-\bar Y_n$ is bounded for $1\leq i\leq n$, and as $n\to\infty$
\be
\label{eq-1107-15}
    \hat{\sigma}_n \rightarrow \sqrt {\bm{\lambda}^T\Sigma\bm{\lambda}/2}\quad a.s.
\en
which implies
\bes
    \sup_{s\in \mathbb{R}}\left| \P^*\left(\sqrt n \frac{\bar{Y}_n^*-\bar{Y}_n}{\hat{\sigma}_n}
    \leq s\right)-\Phi(s)\right|\rightarrow 0,\quad a.s.
\ens
In addition, Taylor expansion and (\ref{eq-1107-15}) yield
\bes
    \sup_{s\in \mathbb{R}}\left| \Phi\ll(\frac{\hat{\sigma}_n}{\sqrt{\bm{\lambda}^T\Sigma\bm{\lambda}/2}}s\rr)-\Phi(s)\right|\rightarrow 0, \quad a.s.
\ens
Now we have proved (\ref{eq-1107-13}), then the conclusion follows.
\end{proof}

\section{An application to the bootstrap projection depth weighted mean}\label{Sec:PWM}
\paragraph{}

In this section, we apply the previous results to obtain the weak Bahadur representation of the bootstrap sample projection depth weighted mean, including the famous Stahel-Donoho location estimator as its special case, described in \cite{Zuo2010}.

Following by \cite{ZCH2004}, the projection depth weighted mean is defined as
\bes
    P\!W\!\!M(F)=\frac{\int_{-\infty}^{\infty} xw(PD(x,F))dF(x)}{\int_{-\infty}^{\infty} w(PD(x,F))dF(x)}
\ens
where $w(t)$ is a weight function on $[0, 1]$, $PD(x, F)=\frac1{1+|x-v|/\xi}$ with $v$ and $\xi$ standing for the median and MAD, respectively. By replacing $F$ with $F_n$ and $F_n^*$, respectively, we get the sample and bootstrap versions of P\!W\!M, i.e.
\bes
    P\!W\!\!M(F_n)=\frac{\sum_{t=1}^n w_i X_i}{\sum_{t=1}^n w_i}, ~\text{and}~P\!W\!\!M(F_n^*)=\frac{\sum_{t=1}^n w_i^* X_i^*}{\sum_{t=1}^n w_i^*},
\ens
with
\bes
    w_i=w(PD(X_i,F_n))=w\left(\frac1{1+|X_i- \text{Med}_n|/ \text{MAD}_n}\right)
\ens
and
\bes
    w_i^*=w(PD(X_i^*,F_n^*))=w\left(\frac1{1+|X_i^* - \text{Med}_n^*|/ \text{MAD}_n^*}\right).
\ens

The follows theorem states the weak Bahadur representation of $P\!W\!\!M(F_n^*)$.

\begin{theorem}\label{th:04}
Suppose $F$ is continuous in the neighborhoods of $v\pm \xi$, and differentiable at $v$ and $v\pm \xi$, with $F'(v)>0$ and $G'(\xi)=F'(v-\xi)+F'(v+\xi)>0$, $w(t)$ is continuously differentiable with $w(0)=0$.
Then as $n\to\infty$, we have
\bes
    P\!W\!\!M(F_n^*)-P\!W\!\!M(F)=\frac1n \sum_{t=1}^n [K(X_i^*)-\E K(X_i^*)] + o_p\ll(\frac1{\sqrt n}\rr),
\ens
where
\bes
    K(x)=\frac{\int_{-\infty}^{\infty} [y-P\!W\!\!M(F)]w'(PD(y,F))f(y,x) dF(y)+[x-P\!W\!\!M(F)]w(PD(x,F))}{\int_{-\infty}^{\infty} w(PD(x,F)) dF(x)}
\ens
with
\begin{eqnarray}
\label{eq-1827-1}\begin{aligned}
    f(x,y)=&\frac{|x-v|}{(\xi+|x-v|)^2}\frac{\frac12-I(v-\xi<y\leq v+\xi)}{G'(\xi)}\\
    & +\left[\frac{|x-v|(F'(v+\xi)-F'(v-\xi))}{G'(\xi)(\xi+|x-v|)^2}+\frac{\xi \text{sign}(x-v)}{(\xi+|x-v|)^2}\right]\frac{\frac12-I(y\leq v)}{F'(v)}.
\end{aligned}\end{eqnarray}
\end{theorem}

\begin{proof}[Proof of Theorem \ref{th:04}]
By Theorem \ref{th:03}, it is easy to check that
\begin{eqnarray*}
    &&PD(x, F_n^*)-PD(x, F)\\
    &&=\frac1{1+|x-\text{Med}_n^*|/\text{MAD}_n^*}-\frac1{1+|x-v|/\xi}\\
    &&=\frac{(\text{MAD}_n^*-\xi)|x-v|+\xi(|x-v|-|x-\text{Med}_n^*|)}{(\text{MAD}_n^*+|x-\text{Med}_n^*|)(\xi+|x-v|)}\\
    &&=\frac{|x-v|}{(\xi+|x-v|)^2} (\text{MAD}_n^*-\xi)+\frac{\xi \text{sign}(x-v)}{(\xi+|x-v|)^2}(\text{Med}_n^*-v)+o_p\ll(\frac1{\sqrt n}\rr).
\end{eqnarray*}
From Theorem \ref{th:add2} and Lemma 3.4 of \cite{Zuo2015}, we have
\bes
    \text{Med}_n^*-v=\frac1n \sum_{i=1}^n \frac{\frac12-I(X_i^*\leq v)}{F'(v)}+o_p\ll(\frac1{\sqrt n}\rr)
\ens
and
\bes
    \text{MAD}_n^*-\xi&=&\frac1n \sum_{i=1}^n \frac{\frac12-I(v-\xi<X_i^*\leq v+\xi)}{G'(\xi)}\\
    &&+\frac{F'(v+\xi)-F'(v-\xi)}{G'(\xi)}(\text{Med}_n^*- v) +o_p\ll(\frac1{\sqrt n}\rr),
\ens
which imply that
\begin{equation}\label{eq-1828-1}
    PD(x, F_n^*)-PD(x, F)=\frac1n \sum_{i=1}^n f(x,X_i^*)+o_p\ll(\frac1{\sqrt n}\rr),
\end{equation}
where $f(x, y)$ is defined by (\ref{eq-1827-1}).
Without loss of generality, we assume $P\!W\!\!M(F)=0$, then
\begin{eqnarray}\label{eq-1828-2}
\begin{aligned}
    &\sqrt n \int_{-\infty}^{\infty} x w(PD(x,F_n^*))dF_n^*(x)\\
    &=\sqrt n \int_{-\infty}^{\infty} x w(PD(x,F_n^*))dF_n^*(x)-\sqrt n \int_{-\infty}^{\infty} xw(PD(x,F))dF(x)\\
    &= \int_{-\infty}^{\infty} x w'(\theta_n^*(x))\{\sqrt n[PD(x,F_n^*)-PD(x,F)]\}dF_n^*(x)\\
    &\quad-\int_{-\infty}^{\infty} x w(PD(x,F))d\{\sqrt n [F_n^*(x)-F(x)]\},
\end{aligned}\end{eqnarray}
where $\theta_n^*(x)$ is between $PD(x,F_n^*)$ and $PD(x,F)$, hence satisfying
\bes
    \sup_{x\in\mathbb{R}}|\theta_n^*(x)-PD(x,F)|=O_p\ll(\frac{1}{\sqrt n}\rr),~ n\to\infty.
\ens
Note that as $n\to\infty$
\bes
    \sup_{x\in\mathbb{R}} \left|x[PD(x,F_n^*)-PD(x,F)]\right| = O_p\ll(\frac{1}{\sqrt n}\rr),
\ens
which implies that
\bes
    \left|\int_{-\infty}^{\infty} x\{w'(\theta_n^*(x))-w'(PD(x,F))\}\{\sqrt n[PD(x,F_n^*)-PD(x,F)]\}dF_n^*(x)\right|=o_p(1)
\ens
and
\begin{eqnarray*}
    &&\left|\int_{-\infty}^{\infty} xw'(PD(x,F))\{\sqrt n[PD(x,F_n^*)-PD(x,F)]\}d\{F_n^*(x)-F(x)\}\right|\\
    &&=\frac1{\sqrt n}\left|\int_{-\infty}^{\infty} xw'(PD(x,F))\{\sqrt n[PD(x,F_n^*)-PD(x,F)]\}d\{\sqrt n[F_n^*(x)-F(x)]\}\right|\\
    &&=o_p(1).
\end{eqnarray*}
Hence we have as $n\to\infty$
\be
\label{eq-1828-3}
    &&\int_{-\infty}^{\infty} x w'(\theta_n^*(x))\{\sqrt n[PD(x,F_n^*)-PD(x,F)]\}dF_n^*(x)\\
    &&=\int_{-\infty}^{\infty} x w'(PD(x,F))\{\sqrt n[PD(x,F_n^*)-PD(x,F)]\}dF(x)+o_p(1).\nonumber
\en
It follows from (\ref{eq-1828-3}) and Fubini's theorem that
\begin{eqnarray}\label{eq-1828-4}
\begin{aligned}
&\int_{-\infty}^{\infty} xw'(\theta_n^*(x))\{\sqrt n[PD(x,F_n^*)-PD(x,F)]\}dF_n^*(x)\\
&=\int_{-\infty}^{\infty} xw'(PD(x,F))\{\sqrt n[PD(x,F_n^*)-PD(x,F)]\}dF(x)+o_p(1)\\
&=\int_{-\infty}^{\infty} xw'(PD(x,F))\left(\int_{-\infty}^{\infty} f(x,y)d\{\sqrt n[F_n^*(y)-F(y)]\}\right)dF(x)+o_p(1)\\
&=\int_{-\infty}^{\infty}\int_{-\infty}^{\infty} yw'(PD(y,F)) f(y,x)dF(y)d\{\sqrt n[F_n^*(x)-F(x)]\}+o_p(1).
\end{aligned}\end{eqnarray}
Similarly, we can show that
\begin{equation}\label{eq-1828-5}
\int_{-\infty}^{\infty} w(PD(x, F_n^*)) dF_n^*(x)=\int_{-\infty}^{\infty} w(PD(x,F))dF(x)+o_p(1).
\end{equation}
Then the desired result follows from (\ref{eq-1828-1}), (\ref{eq-1828-2}), (\ref{eq-1828-4}), (\ref{eq-1828-5}) and Slutsky's theorem.
\end{proof}

\begin{corollary}\label{co:03}
Under the conditions of Theorem \ref{th:04}, we have as $n\to\infty$
\[\sqrt n\lra{P\!W\!\!M(F_n^*)-P\!W\!\!M(F)} \toind N\left(0, ~2\text{var}[K(X)]\right),\]
where\[K(x)=\frac{\int_{-\infty}^{\infty} [y-P\!W\!\!M(F)]w'(PD(y,F))f(y,x) dF(y)+[x-P\!W\!\!M(F)]w(PD(x,F))}{\int_{-\infty}^{\infty} w(PD(x,F)) dF(x)}\]
with $f(x,y)$ defined in Theorem \ref{th:04}.
\end{corollary}

\begin{proof}[Proof of Corollary \ref{co:03}]
The conclusion follows from Theorem \ref{th:04}, Theorem 3.1 of \cite{ZCH2004}, and the arguments in the proof of Theorem \ref{th:03}, and we omit the details.
\end{proof}

\begin{remark}
  To increase the breakdown point robustness of the projection median, \cite{Zuo2003} suggested to use a modified sample $\textrm{MAD}$, i.e.,
  \bes
      \textrm{MAD}^{*}_{nk}=\frac{1}{2}\left(W^{*}_{\ll \lfloor\frac{n+k}{2}\rr \rfloor:n}   +W^{*}_{\ll\lfloor\frac{n+k+1}{2}\rr\rfloor:n}\right),
  \ens
for some proper choice of $k=1, \ldots, n-1$, instead, in the definition of the projection depth. If the weighted mean is defined on this depth, the limit distribution of the related bootstrap estimator can be derived in a similar way to those of Theorem \ref{th:01} and Theorem \ref{th:02}.
\end{remark}

\section{Concluding remarks}\label{Sec:conclusion}
\paragraph{}

In this paper, we considered the strong and weak Bahadur representations of the bootstrap MAD, and then used these results to derive the joint limit distribution of the bootstrap sample median and MAD. As an application, we further investigated the weak Bahadur and limit distribution of the bootstrap projection depth weighted mean in one-dimensional space. Being aware that the limit distribution of some inferential estimators/procedures induced from the projection depth, e.g., projection median, is not standard in spaces of dimension greater than 1 due to involving the methodology of projection pursuit. This may hamper their practical applications. Additional bootstrap procedures are needed to obtain the related critical values. We hope the research conducted in the current paper will have the potential to help these studies.

\section*{Acknowledgements}
\paragraph{}

Xiaohui Liu's research was supported by NSF of China (Grant No.11601197, 11461029), China Postdoctoral Science Foundation funded project (2016M600511, 2017T100475), the Postdoctoral Research Project of Jiangxi (2017KY10), NSF of Jiangxi Province (No.20171ACB21030).

\end{document}